\documentclass[12pt]{amsart}
\usepackage{amsfonts}
\usepackage{amssymb}
\usepackage{amsmath}
\usepackage{epsf}
\usepackage{epsfig}
\usepackage{colordvi}
\usepackage{fancybox}
\usepackage[colorlinks=true,urlcolor=blue,linkcolor=blue,pdfborder={0 0 0}]{hyperref}

\newcommand{\R}{\mathbb R}

\newcommand{\E}{\mathbb E}

\newcommand{\const}{\mathrm{const}}

\newtheorem{thm}{Theorem}[section]
\newtheorem{cor}[thm]{Corollary}
\newtheorem{lem}[thm]{Lemma}

\theoremstyle{definition}
\newtheorem{defn}[thm]{Definition}

\theoremstyle{remark}
\newtheorem{remark}{Remark}[section]

\newtheorem{example}{Example}[section]

\newcommand{\ds}{\displaystyle}
\newcommand{\tr}{\mathrm{tr}}

\numberwithin{equation}{section}

\begin{document}

\title[Canonical parameters on marginally trapped surfaces]
{Canonical parameters on marginally trapped surfaces in the Minkowski 4-space}%

\author{Miroslav Maksimovi\'c and Velichka Milousheva}%
\address{University of Pri\v{s}tina in Kosovska Mitrovica, Faculty of Sciences and Mathematics, Kosovska Mitrovica, Serbia; and 
Institute of Mathematics and Informatics, Bulgarian Academy of Sciences, Acad. G. Bonchev Str. bl. 8, 1113, Sofia, Bulgaria}
\address{Institute of Mathematics and Informatics, Bulgarian Academy of Sciences, Acad. G. Bonchev Str. bl. 8, 1113, Sofia, Bulgaria}
\email{miroslav.maksimovic@pr.ac.rs}
\email{vmil@math.bas.bg}


\begin{abstract}

Marginally trapped surfaces are spacelike surfaces in the Minkowski space whose mean curvature vector is lightlike at each point. In general, the 
marginally trapped surfaces are determined by seven functions satisfying  several conditions (differential equations). In the present paper, we introduce special principal parameters, called canonical, and prove that every marginally trapped surface of general type  admits (at least locally) canonical principal parameters which allow us to reduce the number of functions. We prove a Fundamental existence and uniqueness theorem formulated in terms of canonical parameters, which states that every marginally trapped surface of general type is determined up to a motion by three smooth functions satisfying a system of partial differential equations. { The special subclass of  marginally trapped surfaces with parallel mean curvature vector field is described in terms of only one differential equation. }  
{ Finally, the theory is enriched with examples of marginally trapped surfaces for which canonical parameters are determined.}

\end{abstract}

\maketitle




\section{Introduction}

In the classical differential geometry, the problem of determining a regular curve in the space is solved by introducing the concept of a natural parameter -- the curvature and the torsion of the curve, expressed with respect to the natural parameter, uniquely determine the curve up to a position in the space. However, in the case of surfaces,  the situation is significantly more complicated. The problem of determining a surface with fewer conditions than the conditions in the classical Bonnet’s theorem has been considered by mathematicians for many years. This problem was of interest even to O. Bonnet \cite{Bon}, who considered the possibility of reducing the functions to the coefficients of the first fundamental form and the two principal curvatures. Later, the same problem was studied by \'E. Cartan, who also investigated the problem of determining the surface only by the second fundamental form \cite{Cartan-1}. S. P. Finikoff and B. Gambier \cite{Fin-Gam}, and later R. Bryant \cite{Bry} investigated surfaces in the 3-dimensional Euclidean space from the viewpoint of curvature lines and  principal curvatures. 

The problem of finding a minimal number of functions, satisfying a minimal number of conditions (differential equations), that determine the surface up to a motion in the ambient space, is a fundamental problem in the local theory of surfaces both in Euclidean and pseudo-Euclidean spaces. It is known as the Lund-Regge problem \cite{Lund-Reg} and  is solved for some special classes of surfaces using appropriate parameters on the surface. For example,  special principal parameters were introduced for the so-called Weingarten surfaces in \cite{Gan-Mih1}, and these parameters were used to determine the Weingarten surfaces by three functions, satisfying  one partial differential equation (which is equivalent to the Gauss equation). This result was later improved by O. Kassabov \cite{OK} by introducing canonical parameters for an arbitrary (not necessarily Weingarten) smooth surface and showing that the number of functions involved is reduced to two,  moreover these are the two main invariants -- the Gaussian curvature and the mean curvature.

In the last few years, the problem of determining a surface by a minimal number of functions satisfying a minimal number of partial differential equations was attacked for surfaces of co-dimension two in the 4-dimensional spaces $\R^4$,  $\R_1^4$   and  $\R_2^4$. The Lund-Regge problem was solved for some special classes of surfaces, for example, the minimal surfaces and the surfaces with parallel normalized mean curvature vector field.
By use of special geometric parameters on any minimal surface in  $\R^4$,  R. Tribuzy and I. Guadalupe proved that a minimal surface in  $\R^4$  is determined up to a motion by two invariant functions (the Gaussian curvature and the normal curvature) satisfying a system of two PDEs \cite{Trib-Guad}.  A similar result was proved for spacelike and timelike zero mean curvature surfaces in the Minkowski 4-space $\R^4_1$ in \cite{Al-Pal} and \cite{G-M-IJM}, respectively. These classes of surfaces in  $\R^4_1$  admit (at least locally) special isothermal parameters, called \textit{canonical},  such that the two main invariants -- the Gaussian curvature and the normal curvature of the surface satisfy a system of two partial differential equations. 
The same approach was applied to minimal  Lorentz surfaces in the pseudo-Euclidean 4-space with neutral metric  $\R^4_2$ in \cite{AM-JGP}. In all these cases, the number of the invariant functions determining the surfaces and the number of the differential equations are both reduced to two. Moreover, the geometry of the corresponding zero mean curvature
surface is determined by the solutions of the corresponding system of PDEs.

Another class of surfaces for which the number of functions and the number of differential equations can be reduced, is the class of surfaces with parallel normalized mean curvature vector field -- an important class of surfaces both in Riemannian and pseudo-Riemannian geometry. 
In \cite{G-M-Fil}, the surfaces with parallel normalized mean curvature vector field (PNMCVF) in the Euclidean 4-space $\R^4$ and the spacelike PNMCVF-surfaces in the Minkowski 4-space $\R^4_1$ are described  in terms of three  functions satisfying a system of three PDEs. 
A similar result is proved for the class of  timelike PNMCVF-surfaces in $\R^4_1$, see \cite{B-M}. In all these cases, the basic  approach to solve the Lund-Regge problem is based on introducing special geometric parameters  called \textit{canonical parameters}. 

The idea to introduce special parameters was further developed by O. Kassabov and the second author in \cite{K-M-2025}, where they defined  canonical principal parameters for the general class of surfaces in the Euclidean 4-space $\R^4$, which generalize the canonical parameters introduced for minimal surfaces and for PNMCVF-surfaces. They proved the Fundamental existence and uniqueness theorem in terms of canonical parameters, which states that any surface in $\R^4$  is determined up to a motion by four geometrically determined functions satisfying a system of natural PDEs. 
The same approach can be applied to the general class of spacelike surfaces in the Minkowski 4-space $\R^4_1$. 

\vskip 1mm
In the present paper, we consider the class of the so-called  marginally trapped surfaces in the four-dimensional Minkowski space $\R^4_1$, which is a  natural extension of the class of minimal surfaces. A surface in a 4-dimensional spacetime is called \textit{marginally trapped} (or \textit{quasi-minimal}) if its mean curvature vector is lightlike at each point of the surface  \cite{Vran-Rosca, Chen-Garay}. 
The concept of trapped surfaces was first  introduced by Roger Penrose \cite{Pen} and plays an important role in General Relativity for studying global properties of spacetime and also for understanding the evolution of cosmic  black holes. Marginally trapped surfaces have been very actively studied in the last few years. For example, marginally trapped surfaces with
positive relative nullity in Lorentz space forms were classified by Chen and Van der Veken in \cite{Chen-Veken-1}. They also classified marginally
trapped surfaces with parallel mean curvature vector field \cite{Chen-Veken-2}. In a series of papers, S. Haesen and M. Ortega studied marginally trapped surfaces in the Minkowski 4-space which are invariant under spacelike rotations, under  boost transformations (hyperbolic rotations), and under the group of screw rotations (a group of Lorentz rotations with an invariant lightlike direction), see \cite{Haesen-Ort-1,Haesen-Ort-2,Haesen-Ort-3}.

In \cite{G-M-JMP-2012}, G. Ganchev and the second author developed an invariant theory
of marginally trapped surfaces in the four-dimensional Minkowski space $\R^4_1$ based on the principal lines
generated by the second fundamental form. They proved a fundamental theorem of
Bonnet-type stating that each marginally trapped surface in $\R^4_1$  is determined up to a motion by seven invariant functions satisfying some natural conditions.

In the present paper, we define special principal parameters, which we call canonical, on the general class of marginally trapped surface in $\R^4_1$,  and prove that each such surface admits (at least locally) canonical principal parameters.  We prove a Fundamental existence and uniqueness theorem formulated in terms of canonical parameters, which states that any marginally trapped surface of general type in $\R^4_1$ is determined up to a motion by three geometrically determined functions satisfying a system of partial differential equations { (Theorem \ref{eq:MainTheorem}).
The special class of marginally trapped surfaces with parallel mean curvature vector field is described in terms of only one differential equation (Theorem \ref{eq:MainTheoremparallel}). }
In the last section, we give a special family of examples demonstrating our theory.

\section{Preliminaries}

We consider the Minkowski 4-space  $\R^4_1$  endowed with the canonical metric  of signature $(3,1)$, 
given by 
$$g_0 = dx_1^2 + dx_2^2 + dx_3^2 -dx_4^2,$$
where $(x_1; x_2; x_3; x_4)$  is a rectangular coordinate system of $\R^4_1$. As usual, we denote by $\langle , \rangle$ the indefinite inner scalar product with respect to $g_0$. 

A surface $M^2$ in $\R^4_1$ is said to be
\emph{spacelike} if $\langle , \rangle$ induces  a Riemannian
metric $g$ on $M^2$. So, at each point $p$ of a spacelike surface
$M^2$ we have the following decomposition into a tangent and normal space
$$\R^4_1 = T_pM^2 \oplus N_pM^2$$
such  that the restriction of the metric
$\langle , \rangle$  onto the tangent space $T_pM^2$ is of
signature $(2,0)$, and the restriction of the metric $\langle , \rangle$ onto the normal space $N_pM^2$ is of signature $(1,1)$.

Denote by $\nabla'$ and $\nabla$ the Levi-Civita connections on $\R^4_1$ and $M^2$, respectively.
For any tangent vector fields   $x$ and $y$ and any normal vector field $\xi$ we have the following formulas of Gauss and Weingarten:
$$\begin{array}{l}
\vspace{2mm}
\nabla'_xy = \nabla_xy + \sigma(x,y);\\
\vspace{2mm}
\nabla'_x \xi = - A_{\xi} x + D_x \xi,
\end{array}$$
where $\sigma$ is the second fundamental tensor, $D$ -- the normal connection 
and  $A_{\xi}$ -- the shape operator with respect to $\xi$.
The mean curvature vector  field $H$ of the surface $M^2$ is defined as
$H = \ds{\frac{1}{2}\,  \tr\, \sigma}$.

A surface $M^2$  is called \textit{minimal} if its mean curvature vector vanishes identically, i.e.
$H = 0$.  A surface $M^2$ in $\R^4_1$ is called \textit{marginally trapped} (or \textit{quasi-minimal}) if its mean curvature vector is lightlike at each point, i.e. $H \neq  0$ and $\langle H, H \rangle = 0$.  

Let $M^2: z = z(u,v), \, \, (u,v) \in {\mathcal D}$ (${\mathcal D} \subset \R^2$) be a local parametrization of a marginally trapped surface in $\R^4_1$ and denote by $E$, $F$, $G$ the coefficients of the 
first fundamental form of $M^2$.
In \cite{G-M-JMP-2012} it is proved that in a neighbourhood of each  non-inflection point $p \in M^2$ there exists a geometrically determined orthonormal frame field $\{x,y,n_1,n_2\}$, such that 
$x = \ds{\frac{z_u}{\sqrt{E}}}$  and $y = \ds{\frac{z_v}{\sqrt{G}}}$ are collinear with the principal directions, $n_1 = H$ and $n_2$ is a normal vector field determined by the conditions
$$\langle n_1, n_1 \rangle = 0; \quad \langle n_2, n_2 \rangle = 0; \quad \langle n_1, n_2 \rangle = -1.$$
With respect to this frame field  the
following Frenet-type derivative formulas  hold true:
\begin{equation} \label{Eq-5}
	\begin{array}{ll}
		\vspace{2mm} \nabla'_xx=\quad \quad \quad \gamma_1\,y+\,(1+\nu)\,n_1;
		& \qquad
		\nabla'_x n_1= \quad\quad \quad \quad  \mu\,y+\beta_1\,n_1;\\
		\vspace{2mm} \nabla'_xy=-\gamma_1\,x\quad \quad \; + \,\quad  \lambda\,n_1
		\; + \mu\,n_2;  & \qquad
		\nabla'_y n_1=\mu \,x \quad\quad \quad \quad +\beta_2\,n_1;\\
		\vspace{2mm} \nabla'_yx=\quad\quad \;\, -\gamma_2\,y  + \quad \lambda\,n_1
		\; +\mu\,n_2;  & \qquad
		\nabla'_xn_2= (1+\nu) \, x + \lambda \,y \quad \quad  -\beta_1\,n_2;\\
		\vspace{2mm} \nabla'_yy=\;\;\gamma_2\,x \quad\quad\;\;\,
		+(1-\nu)\,n_1; & \qquad \nabla'_y n_2= \lambda \, x  + (1-\nu) \,y
		\quad \quad -\beta_2\,n_2,
	\end{array}
\end{equation}
where $\nu, \lambda, \mu, \gamma_1, \gamma_2, \beta_1, \beta_2$ are smooth functions determined by the geometric frame field $\{x,y,n_1,n_2\}$. These functions are called \textit{geometric functions} of the surface and their geometric meaning is described in \cite{G-M-JMP-2012}. Note that $\mu \neq 0$\footnote{In the case  $\mu = 0$ both $K$ and $\varkappa$ are zero at each point and the surface consists of inflection points, so we exclude this case.} at each point and the main invariants of the surface -- the Gauss curvature $K$ and the curvature of the normal connection $\varkappa$ are expressed by the functions $\lambda$, $\mu$ and $\nu$ as follows:
\begin{equation} \label{eq:K_kappa}
K=2\lambda\mu; \qquad  \varkappa=-2\mu\nu.
\end{equation} 
The functions $\gamma_1$ and $\gamma_2$  are expressed in terms of the coefficients of the first fundamental form by the formulas
\begin{equation*} 
	\gamma_1=-\frac{(\sqrt E)_v}{\sqrt E\sqrt G},  \qquad \gamma_2=-\frac{(\sqrt G)_u}{\sqrt E\sqrt G},
\end{equation*} 
or equivalently
\begin{equation} \label{eq:gamma1}
	\gamma_1=-\frac{1}{\sqrt G} (\ln \sqrt E)_v,  \qquad \gamma_2=-\frac{1}{\sqrt E} (\ln \sqrt G)_u.
\end{equation}

 In  \cite{G-M-JMP-2012}, it was proved that the seven functions $\nu, \lambda, \mu, \gamma_1, \gamma_2, \beta_1, \beta_2$  satisfy the following conditions
 \begin{equation} \label{eq:BasicSystem}
 	\begin{array}{l}
 		\vspace{2mm} 2\mu\, \gamma_2 + \mu\,\beta_1 =
 		\displaystyle{\frac{\mu_u}{\sqrt{E}}  };\\
 		\vspace{2mm} 2\mu\, \gamma_1+ \mu\,\beta_2 =
 		\displaystyle{\frac{\mu_v}{\sqrt{G}} };\\
 		\vspace{2mm}
 		2\lambda\, \mu =
 		\displaystyle{\frac{(\gamma_2)_u}{\sqrt{E}}
 			+ \frac{(\gamma_1)_v}{\sqrt{G}} - \left((\gamma_1)^2 + (\gamma_2)^2\right)};\\
 		\vspace{2mm} 2\lambda\, \gamma_2 - 2\nu\, \gamma_1 - \lambda\,\beta_1 + (1+ \nu)\,\beta_2 =
 		\displaystyle{\frac{\lambda_u}{\sqrt{E}}  - \frac{\nu_v}{\sqrt{G}}};\\
 		\vspace{2mm} 2\lambda\, \gamma_1 + 2\nu\, \gamma_2 + (1 - \nu)\,\beta_1 - \lambda\,\beta_2 =
 		\displaystyle{\frac{\nu_u}{\sqrt{E}} + \frac{\lambda_v}{\sqrt{G}}};\\
 		\gamma_1\,\beta_1 - \gamma_2\,\beta_2 + 2 \nu\,\mu  =
 		\displaystyle{ - \frac{(\beta_2)_u}{\sqrt{E}} +
 			\frac{(\beta_1)_v}{\sqrt{G}}},
 	\end{array}
 \end{equation}
and determine the marginally trapped surface  up to a motion in $\R^4_1$ (see \cite{G-M-JMP-2012}, Theorem 4.1),  in the case $\mu_u \mu_v \neq 0$. This is the Bonnet-type Fundamental theorem for this class of marginally trapped surfaces, which gives the number of functions and the number of differential equations determining the surface   with respect to arbitrary principal parameters.

In the next section, we will introduce special principal parameters, which  allow us to reduce up to three the number of functions  and the  number of PDEs determining { this} class of marginally trapped surfaces.

\section{Canonical parameters on a marginally trapped surface}

Let $M^2$ be a marginally trapped surface parametrized by arbitrary principal parameters  and $\mu_u \mu_v \neq 0$. Then formulas \eqref{Eq-5} hold true.
The first two equations of  (\ref{eq:BasicSystem}) and  (\ref{eq:gamma1}) imply 
\begin{equation*} \label{eq:3.1a}
	\begin{array}{l}
		\vspace{2mm}
		\ds -2\mu\frac{1}{\sqrt{E}}(\ln \sqrt{ G})_u+\mu\beta_1=\frac{\mu_u}{\sqrt E} \ ; \\
		\vspace{2mm}
		\ds	-2\mu\frac{1}{\sqrt{G}}(\ln \sqrt{ E})_v+\mu\beta_2=\frac{\mu_v}{\sqrt G}.
	\end{array}
\end{equation*}
The last equalities imply that the functions $\beta_1$ and $\beta_2$ are expressed in terms of the function $\mu$ and the coefficients of the first fundamental form as follows:
\begin{equation} \label{eq:beta}
	\beta_1=\frac{1}{\sqrt E} (\ln |\mu|)_u + \frac{2}{\sqrt{E}}(\ln \sqrt{ G})_u,  \qquad \beta_2=\frac{1}{\sqrt G} (\ln |\mu|)_v + \frac{2}{\sqrt{G}}(\ln \sqrt{ E})_v.
\end{equation} 

Now, using  \eqref{eq:gamma1} and  \eqref{eq:beta}, from the fourth and fifth equations of \eqref{eq:BasicSystem} we obtain  the following relations:
\begin{equation}\label{eq:mainsystem}
\ds 2(2\nu+1) \frac{(\ln \sqrt{ E})_v}{\sqrt{G}}	-4\lambda\, \frac{(\ln \sqrt{ G})_u}{\sqrt{E}}  =
	\displaystyle{\frac{\lambda_u}{\sqrt{E}}  - \frac{\nu_v}{\sqrt{G}} + \frac{\lambda\, (\ln |\mu|)_u}{\sqrt E}  - \frac{(1+\nu)(\ln |\mu|)_v}{\sqrt G}},
\end{equation}
\begin{equation}\label{eq:mainsystemb}
		\ds	-4\lambda\, \frac{(\ln \sqrt{E})_v}{\sqrt{G}} - 2(2\nu-1) \frac{(\ln \sqrt{G})_u}{\sqrt{E}}  =
		\displaystyle{\frac{\nu_u}{\sqrt{E}} + \frac{\lambda_v}{\sqrt{G}}  - \frac{(1-\nu)(\ln |\mu|)_u}{\sqrt E}  + \frac{\lambda (\ln |\mu|)_v}{\sqrt G}}.
\end{equation}
We consider the above two equations as a system with respect to $(\sqrt E)_v$ and $(\sqrt G)_u$. In the general case, when $4\nu^2+4\lambda^2-1\ne0$, 
equalities  \eqref{eq:mainsystem} and \eqref{eq:mainsystemb} imply the next system of partial differential equations 
\begin{equation*} 
	\begin{array}{l}
		\vspace{2mm}
		\left(\sqrt E\right)_v=\phi_1\sqrt E+\phi_2\sqrt G \ ; \\
		\vspace{2mm}
		\left(\sqrt G\right)_u=\phi_3\sqrt E+\phi_4\sqrt G \ ,
	\end{array}
\end{equation*} 
which can be written also as follows
\begin{equation} \label{eq:3.3-v2}
	\begin{array}{l}
		\vspace{2mm}
		\ds\left(\ln\sqrt E\right)_v=\phi_1+\phi_2\frac{\sqrt G}{\sqrt E} \ ; \\
		\vspace{2mm}
		\ds\left(\ln\sqrt G\right)_u=\phi_3\frac{\sqrt E}{\sqrt G}+\phi_4 \ ,
	\end{array}   
\end{equation} 
where the four functions $\phi_i$, $i=1,2,3,4$, are given by 
\begin{equation} \label{eq:functions}
	\begin{array}{l}
		\vspace{3mm}
		\phi_1 \!= \! \ds{-\frac{\mu(\lambda^2+\nu^2-\nu)_v + (2(\lambda^2+\nu^2)+\nu-1)\mu_v }{2\mu(4\nu^2+4\lambda^2-1) } } \ ; \\
		\vspace{3mm}
		\phi_2 \!=  \! \ds{\frac{2\mu(\lambda_u\nu-\lambda\nu_u) +\lambda\mu_u-\lambda_u\mu }{2\mu(4\nu^2+4\lambda^2-1) } } \ ;\\
		\vspace{3mm}
		\phi_3 \!= \! \ds{\frac{2\mu(\lambda\nu_v -\lambda_v\nu) +\lambda\mu_v-\lambda_v\mu }{2\mu(4\nu^2+4\lambda^2-1) } }\ ; \\
		\vspace{3mm}
		\phi_4 \!= \! \ds{-\frac{\mu(\lambda^2+\nu^2+\nu)_u + (2(\lambda^2+\nu^2)-\nu-1)\mu_u  }{2\mu(4\nu^2+4\lambda^2-1) } } \ .
	\end{array}
\end{equation}

\begin{remark}
Note that from equalities \eqref{eq:K_kappa} it follows that the condition $4\nu^2+4\lambda^2-1 = 0$ is equivalent to the following condition on the main invariants of the surface:
$$K^2+\varkappa^2 =\mu^2.$$
\end{remark}

Further, we shall consider the general case: $K^2+\varkappa^2 \neq \mu^2$ and these surfaces we will call \textit{marginally trapped surfaces of general type}. The special case $K^2+\varkappa^2 =\mu^2$ will be studied separately.

\vskip 2mm
Having in mind system of equations (\ref{eq:3.3-v2}), we can easily prove the following statement.

\begin{lem} \label{lemma:ConstantFunctions} The function $\sqrt E\, e^{-\ds\int_{v_0}^v\left(\phi_1+\phi_2\frac{\sqrt G}{\sqrt E}\right)dv} $ does not depend on the parameter $v$ and the function 
	$\sqrt G\, e^{-\ds\int_{u_0}^u\left(\phi_3\frac{\sqrt E}{\sqrt G}+\phi_4\right)du}$
	does not depend on the parameter $u$.
\end{lem}

From Lemma \ref{lemma:ConstantFunctions} it follows that if $(u_0,v_0)$ is an arbitrary point  and $c_1$, $c_2$ are constants, then 
there exist functions $\varphi(u)$ and $\psi(v)$ which are given by
\begin{equation} \label{eq:phi-psi}
	\begin{array}{l} 
		\vspace{2mm}
		\varphi(u) = \ds\sqrt Ee^{-\ds\int_{v_0}^v\left(\phi_1+\phi_2\frac{\sqrt G}{\sqrt E}\right)dv
			-\int_{u_0}^u\left(\phi_3\frac{\sqrt E}{\sqrt G}+\phi_4\right)(u,v_0)du+c_1}\ ;  \\
		\vspace{2mm}
		\psi(v) = \ds\sqrt Ge^{-\ds\int_{u_0}^u\left(\phi_3\frac{\sqrt E}{\sqrt G}+\phi_4\right)du
			-\int_{v_0}^v\left(\phi_1+\phi_2\frac{\sqrt G}{\sqrt E}\right)(u_0,v)dv+c_2} \ .
	\end{array}   
\end{equation}

We introduce the notion of canonical parameters on a marginally trapped surface by the following definition.

\begin{defn} \label{D:def-can}
Let $M^2$ be a marginally trapped surface in $\R^4_1$ parametrized by principal parameters $(u,v)$ and  $K^2+\varkappa^2 \neq \mu^2$, { $\mu_u \mu_v \neq 0$.} If the functions $\varphi(u)$ and $\psi(v)$ defined by \eqref{eq:phi-psi} are equal to 1 for some values of $u_0$, $v_0$, $c_1$, $c_2$, then we say that the parameters $(u,v)$ are {\it canonical principal parameters} of the surface.
\end{defn}

Now we will prove the following theorem for canonical parameters of marginally trapped surfaces in $\R^4_1$.

\begin{thm} \label{eq:ExistenceCanonicalParameters}
	Each marginally trapped surface with  $K^2+\varkappa^2 \neq \mu^2$ locally admits canonical principal parameters.
\end{thm}

\begin{proof}
	For a given  pair of principal parameters $(u,v)$ and an arbitrary point $(u_0,v_0)$, we introduce new parameters $\bar u$, $\bar v$ defined by
	\begin{equation*} 
	\begin{array}{l}
		\bar u=\ds\int_{u_0}^u\sqrt Ee^{-\ds\int_{v_0}^v\left(\phi_1+\phi_2\frac{\sqrt G}{\sqrt E}\right)dv
			-\int_{u_0}^u\left(\phi_3\frac{\sqrt E}{\sqrt G}+\phi_4\right)(u,v_0)du+c_1}  du +u_0 \vspace{2mm} \ ;\\
		\bar v=\ds\int_{v_0}^v\sqrt Ge^{-\ds\int_{u_0}^u\left(\phi_3\frac{\sqrt E}{\sqrt G}+\phi_4\right)du
			-\int_{v_0}^v\left(\phi_1+\phi_2\frac{\sqrt G}{\sqrt E}\right)(u_0,v)dv+c_2}  dv +v_0 \ ,
	\end{array}
\end{equation*} 
	or equivalently,
	\begin{equation} \label{eq:defCanonParam1}
	\bar u=\int_{u_0}^u \varphi(u) du +u_0 \ ;   \qquad  \bar v=\int_{v_0}^v \psi(v) dv +v_0 ,
		\end{equation}  
	where the functions $\varphi(u)$ and $\psi(v)$ are given by \eqref{eq:phi-psi}. Therefore,  $\bar u=\bar u(u)$ and $\bar v=\bar v(v)$ hold, which means that the new  parameters 
	$(\bar u, \bar v)$ are also principal. Obviously, $\bar u_0=\bar u(u_0)=u_0$, 
	$\bar v_0=\bar v(v_0)=v_0$. Differentiating  equations \eqref{eq:defCanonParam1} gives
	\begin{equation*} 
		\begin{array}{l} 
			\vspace{2mm}
	\bar u_u = \varphi(u) =\sqrt Ee^{-\ds\int_{v_0}^v\left(\phi_1+\phi_2\frac{\sqrt G}{\sqrt E}\right)dv
		-\int_{u_0}^u\left(\phi_3\frac{\sqrt {E}}{\sqrt {G}}+\phi_4\right)(u,v_0)du+c_1} \ ;\\
			\vspace{2mm}
	\bar u_v = 0 \ ;\\
	\vspace{2mm}
	\bar v_u = 0 \ ;\\
	\vspace{2mm}
	\bar v_v = \psi(v) =\sqrt Ge^{-\ds\int_{u_0}^u\left(\phi_3\frac{\sqrt E}{\sqrt G}+\phi_4\right)du
		-\int_{v_0}^v\left(\phi_1+\phi_2\frac{\sqrt G}{\sqrt E}\right)(u_0,v)dv+c_2} \ .
	\end{array}   
	\end{equation*}

	It can easily be checked that  under the change of the parameters $\bar u=\bar u(u)$, $\bar v= \bar v(v) $ we have the relations:
	$$\sqrt {E} = \sqrt {\bar E} \, {\bar u}_u; \quad \sqrt {G} = \sqrt {\bar G} \, {\bar v}_v,$$
	from where it follows that
	\begin{equation*} 
		\begin{array}{ll} 
			\vspace{2mm}
			\phi_1 (u,v) = {\bar \phi_1} (\bar u(u),\bar v(v)) \, {\bar v}_v;  & \quad \phi_2 (u,v) = {\bar \phi_2} (\bar u(u),\bar v(v)) \, {\bar u}_u;  \\
			\vspace{2mm}
			\phi_3 (u,v) = {\bar \phi_3} (\bar u(u),\bar v(v)) \, {\bar v}_v;  & \quad \phi_4 (u,v) = {\bar \phi_4} (\bar u(u),\bar v(v)) \, {\bar u}_u.
		\end{array}   
	\end{equation*}
The last equations imply 
	$$
	\ds (\phi_1+\phi_2\frac{\sqrt {G}}{\sqrt {E}})(u,v)=(\bar \phi_1+\bar \phi_2\frac{\sqrt {\bar G}}{\sqrt {\bar E}})(\bar u(u),\bar v(v))\frac{d\bar v}{dv} \ ;
	$$
	$$
	\ds (\phi_3\frac{\sqrt {E}}{\sqrt {G}}+\phi_4)(u,v)=(\bar \phi_3\frac{\sqrt {\bar E}}{\sqrt {\bar G}}+\bar \phi_4)(\bar u(u),\bar v(v))\frac{d\bar u}{du} \ .
	$$
By calculations we further obtain
	\begin{equation}\label{eq:3.6}
		\begin{split}
			& e^{\ds\int_{v_0}^v\left(\phi_1+\phi_2\frac{\sqrt G}{\sqrt E}\right)(u,v)dv
				+\int_{u_0}^u\left(\phi_3\frac{\sqrt {E}}{\sqrt {G}}+\phi_4\right)(u,v_0)du} = \\
			& =e^{\ds\int_{v_0}^v\left(\bar \phi_1+\bar \phi_2\frac{\sqrt {\bar G}}{\sqrt {\bar E}}\right)(\bar u(u),\bar v(v))\frac{d\bar v}{dv}{dv}
				+\int_{u_0}^u\left(\bar \phi_3\frac{\sqrt {\bar E}}{\sqrt {\bar G}}+\bar \phi_4\right)(\bar u(u),\bar v(v_0))\frac{d\bar u}{du}du} = \\
			& =e^{\ds\int_{\bar v_0}^{\bar v}\left(\bar \phi_1+\bar \phi_2\frac{\sqrt {\bar G}}{\sqrt {\bar E}}\right)(\bar u,\bar v)d\bar v
				+\int_{\bar u_0}^{\bar u}\left(\bar \phi_3\frac{\sqrt {\bar E}}{\sqrt {\bar G}}+\bar \phi_4\right)(\bar u,\bar v_0)d\bar u} \ .
		\end{split}
	\end{equation}
		Using the previous results we have
	$$
	\sqrt{\bar E}=\frac{\sqrt{E}}{\bar u_u}=e^{\ds\int_{v_0}^v\left(\phi_1+\phi_2\frac{\sqrt G}{\sqrt E}\right)dv
		+\int_{u_0}^u\left(\phi_3\frac{\sqrt {E}}{\sqrt {G}}+\phi_4\right)(u,v_0)du-c_1}
	$$
	
	$$
	=e^{\ds\int_{\bar v_0}^{\bar v}\left(\bar \phi_1+\bar \phi_2\frac{\sqrt {\bar G}}{\sqrt {\bar E}}\right)(\bar u,\bar v)d\bar v
		+\int_{\bar u_0}^{\bar u}\left(\bar \phi_3\frac{\sqrt {\bar E}}{\sqrt {\bar G}}+ \bar \phi_4\right)(\bar u,\bar v_0)d\bar u-c_1}
	$$
	and hence $\bar\varphi (\bar u)=1$, where the function $\bar\varphi$ is given by
$$
\ds\bar\varphi (\bar u) = \sqrt{ \bar E} e^{\ds-\int_{\bar{v}_0}^{\bar v}\left(\bar\phi_1+\bar\phi_2\frac{\sqrt{\bar G}}{\sqrt{\bar E}}\right)d\bar v
	-\int_{\bar{u}_0}^{\bar u}\left(\bar\phi_3\frac{\sqrt{\bar E}}{\sqrt{\bar G}}+\bar\phi_4\right)(\bar u,\bar{v}_0)d\bar u+c_1}.
$$

In a  similar way, we also obtain that $\bar\psi(\bar v)=1$, where the function $\bar\psi$ is defined analogously to function $\psi$, but with respect to the parameters $(\bar u,\bar v)$. Thus, according to Definition \ref{D:def-can}, the parameters $(\bar u,\bar v)$ are canonical.
	
\end{proof}

{
	
	\begin{lem}\label{L:phipsi}
		The functions $\varphi$, $\psi$ given by \eqref{eq:phi-psi} can be written in the forms
		\begin{equation}\label{eq:phi-psi2}
			\begin{array}{l}
				\vspace{2mm}
				\varphi(u)=e^{c_1} \sqrt{ \frac{E(u,v_0) G(u_0,v_0)}{G(u,v_0)}}, \\
				\vspace{2mm}
				\psi(v)=e^{c_2} \sqrt{ \frac{G(u_0,v) E(u_0,v_0)}{E(u_0,v)}}.
			\end{array}
		\end{equation}
	\end{lem}
	
	\begin{proof}
		Taking into account the first equation in \eqref{eq:3.3-v2} we have
		\begin{equation*}
			\begin{split}
				& \int_{v_0}^v\left(\phi_1+\phi_2\frac{\sqrt G}{\sqrt E}\right)dv
				+\int_{u_0}^u\left(\phi_3\frac{\sqrt E}{\sqrt G}+\phi_4\right)(u,v_0)du =  \\ &
				\int_{v_0}^v\left(\ln\sqrt E\right)_v dv
				+\int_{u_0}^u\left(\ln\sqrt G\right)_u(u,v_0)du  =  \\ &
				\ln\sqrt{E(u,v)} -\ln\sqrt{E(u,v_0)} + \ln\sqrt{G(u,v_0)} -\ln\sqrt{G(u_0,v_0)},
			\end{split}
		\end{equation*}
		so the function {$\varphi(u)$} given by \eqref{eq:phi-psi} can be written in the form 
		{ $ \varphi(u)=e^{c_1} \sqrt{ \frac{E(u,v_0) G(u_0,v_0)}{G(u,v_0)}}$.
			The second equation in \eqref{eq:phi-psi2} can be } proved similarly.
		
	\end{proof}	
	
	Using \eqref{eq:phi-psi2}, { we can easily obtain the relation between the  constants } $c_1$, $c_2$ and the functions $E$, $G$ {  in the case of canonical parameters. } 
	
	\begin{cor}\label{Cor:Const}
		If the parameters $(u,v)$ are canonical then 
		\begin{equation*}
			c_1 = -\ln\sqrt{E(u_0,v_0)}, \;\; c_2=-\ln\sqrt{G(u_0,v_0)}.
		\end{equation*}
	\end{cor}
	
}

{ Corollary \ref{Cor:Const} shows that the constants $c_1$, $c_2$ are not free. }

{
\begin{lem} \label{L:uniqueness}
 Let $M^2$ be a marginally trapped surface of
 	general type, and let $(u,v)$ and $(\bar u,\bar v)$ be two pairs of
 	canonical principal parameters on $M^2$, such that  $\bar u_0=\bar u(u_0)$ and $\bar v_0=\bar v(v_0)$. Then
 \begin{equation}\label{eq:change}
 \bar u=\pm k\,u+\alpha,\qquad \bar v=\pm l\,v+\beta,
 \qquad \alpha,\beta\in\mathbb R,
 \end{equation}
where
 \[
 k=e^{\,\bar c_1-c_1}=\sqrt{\frac{E(u_0,v_0)}{\bar E(\bar u_0,\bar v_0)}}>0,
 \qquad
 l=e^{\,\bar c_2-c_2}=\sqrt{\frac{G(u_0,v_0)}{\bar G(\bar u_0,\bar v_0)}}>0
 \]
{(up to interchanging the roles of $u$ and $v$). Conversely, every
 	change of parameters of  the form \eqref{eq:change}  carries canonical principal parameters into
 	canonical principal parameters.}
\end{lem}
\begin{proof}
Since both pairs are principal, $\bar u=\bar u(u)$ and
$\bar v=\bar v(v)$, whence $\sqrt E=\sqrt{\bar E}\,|\bar u_u|$ and
$\sqrt G=\sqrt{\bar G}\,|\bar v_v|$. Substituting this into the form 
of $\varphi$ given by Lemma \ref{L:phipsi}, we obtain
$\varphi(u)=e^{c_1-\bar c_1}\,\bar\varphi(\bar u(u))\,|\bar u_u(u)|$.
Since $\varphi\equiv\bar\varphi\equiv1$, it follows that
$|\bar u_u|=e^{\bar c_1-c_1}$ is a positive constant, and analogously for
$|\bar v_v|$. The value of $k$ follows from Corollary \ref{Cor:Const} together with
$E(u_0,v_0)=\bar E(\bar u_0,\bar v_0)\bar u_u^2$. 

The converse follows immediately 
from Lemma \ref{L:phipsi}, since a scaling  $u\mapsto \pm k\,u+\alpha$, $v\mapsto \pm l\,v+\beta$  multiplies the
ratio $E/G$ by the constant $(l/k)^2$ and therefore preserves both conditions
$\sqrt{E/G}(u,v_0)=\mathrm{const}.$ and
$\sqrt{G/E}(u_0,v)=\mathrm{const}$.

\end{proof}

}

{

\begin{remark}
 If both pairs of canonical parameters are normalized by the same constants, i.e. $\bar{c}_1 = c_1$
and $\bar{c}_2 = c_2$, or equivalently, by Corollary \ref{Cor:Const}, if 
$$E(u_0,v_0) = \bar E(\bar u_0,\bar v_0) = e^{-2c_1}, \quad G(u_0,v_0) = \bar G(\bar u_0,\bar v_0) = e^{-2c_2},$$
then $k=l=1$ and 
$$\bar u=\pm u+\alpha, \quad \bar v=\pm v+\beta.$$
In this normalized form, canonical principal parameters are determined uniquely up to the choice of the
origin and of the direction of the parameter lines (plus possibly an interchange of the roles of $u$ and $v$). This is the exact analogue of the natural parameter
on a curve: there, too, the uniqueness $\bar s = \pm s + s_0$ is a consequence of the normalization $|\dot{z}| =1$, and
without that normalization one would obtain $\bar s = \pm k\, s + s_0$ as well.
\end{remark}

Hence, Lemma \ref{L:uniqueness} shows that the canonical principal parameters introduced in this paper are the surface analogue of the arc-length of a curve. We can also call them natural parameters of the surface, but we prefer to use the notion introduced in previous papers about minimal surfaces and surfaces with parallel normalized mean curvature vector field.

}

\section{Fundamental Theorem in terms of canonical parameters}

In this section, we will consider $(u,v)$ to be canonical principal parameters. In accordance with Definition \ref{D:def-can}, we have
$$
\ds\sqrt E=e^{\ds\int_{v_0}^v\left(\phi_1+\phi_2\frac{\sqrt G}{\sqrt E}\right)dv
	+\int_{u_0}^u\left(\phi_3\frac{\sqrt E}{\sqrt G}+\phi_4\right)(u,v_0)du-c_1},
$$
$$
\ds\sqrt G=e^{\ds\int_{u_0}^u\left(\phi_3\frac{\sqrt E}{\sqrt G}+\phi_4\right)du
	+\int_{v_0}^v\left(\phi_1+\phi_2\frac{\sqrt G}{\sqrt E}\right)(u_0,v)dv-c_2},
$$
and hence 
$$
\ds\sqrt E(u,v_0)=e^{\ds\int_{u_0}^u\left(\phi_3\frac{\sqrt E}{\sqrt G}+\phi_4\right)(u,v_0)du-c_1},
$$
$$
\ds\sqrt G(u,v_0)=e^{\ds\int_{u_0}^u\left(\phi_3\frac{\sqrt E}{\sqrt G}+\phi_4\right)(u,v_0)du-c_2}.
$$
By dividing the last two equalities we get the relation
$$
\left(\ds\sqrt {\frac{E}{G}}\right)(u,v_0)=e^{c_2-c_1}=:c= const,
$$
which gives
$$
\ds\sqrt E(u,v_0)=e^{\ds\int_{u_0}^u\left(c\phi_3+\phi_4\right)(u,v_0)du-c_1}.
$$
Analogously, we obtain
$$
\ds\sqrt G(u_0,v)=e^{\ds\int_{v_0}^v\left(\phi_1+\frac 1c\phi_2\right)(u_0,v)dv-c_2}.
$$
For the purposes of our further work, we introduce the following functions
\begin{equation} \label{eq:g1&g2}
	\begin{array}{l}
		\vspace{2mm}
		\ds g_1(u)=e^{\ds\int_{u_0}^u\left(c\phi_3+\phi_4\right)(u,v_0)du-c_1}, \\
		\vspace{2mm}
		\ds g_2(v)=e^{\ds\int_{v_0}^v\left(\phi_1+\frac 1c\phi_2\right)(u_0,v)dv-c_2}.
	\end{array}
\end{equation}

Now, we shall prove the following Bonnet-type fundamental theorem for marginally trapped surfaces in terms of canonical principal parameters.

\begin{thm} \label{eq:MainTheorem}
	Let $\nu(u,v)$, $\lambda(u,v)$, and  $\mu(u,v)$, $\mu_u\mu_v\neq0$, $4\nu^2+4\lambda^2-1 \neq 0$, be smooth functions defined in a domain ${\mathcal D} \subset \R^2$ and 
	$\phi_i(u,v)$, $i=1,2,3,4$ be defined by \eqref{eq:functions}. Let $\Phi(u,v) \neq 0$, $\Psi(u,v)\neq 0$ be a solution to the Cauchy problem
	\begin{equation*} 
		\begin{array}{l}
			\vspace{2mm}
			\ds\Phi_v=\phi_1\Phi+\phi_2\Psi;  \\ 
			\vspace{2mm}
			\ds\Psi_u   =\phi_3\Phi+\phi_4\Psi; \\
		\end{array}
		\qquad \Phi(u,v_0)=g_1(u); \;\; \Psi(u_0,v)=g_2(v),
	\end{equation*} 
	where $g_1(u)$ and $g_2(v)$ are defined by \eqref{eq:g1&g2}, and let the following equations also hold
	\begin{equation} \label{eq:4.3} 
		\begin{array}{ll}
			\vspace{2mm}
			\ds2\lambda\mu  = & \ds-\frac1{\Phi \Psi}
			\left( \left(\frac{\Phi_v}{\Psi}\right)_v+\left(\frac{\Psi_u}{\Phi}\right)_u \right);   \\
			\vspace{2mm}
			2\nu\mu  = &
			\ds \frac{2}{\Phi\Psi} \left(\ln\frac{|\Psi|}{|\Phi|} \right)_{uv}  .
		\end{array}
	\end{equation}

	Then, there exists a unique (up to a position in $\mathbb R^4_1$) marginally trapped surface of general type parametrized by canonical 
	principal parameters $(u,v)$ with geometric functions $\nu(u,v)$, $\lambda(u,v)$, and $\mu(u,v)$.
\end{thm}

\begin{proof}
	Let $\nu(u,v)$, $\lambda(u,v)$, and  $\mu(u,v)$ be smooth functions, $\phi_i(u,v)$, $i=1,2,3,4$ be given by equations \eqref{eq:functions}, and $g_1(u)$ and $g_2(v)$ be defined by \eqref{eq:g1&g2}.  Let us consider  the following Cauchy problem
	\begin{equation} \label{eq:4.1}
		\begin{array}{l}
			\vspace{2mm}
			\ds\Phi_v=\phi_1\Phi+\phi_2\Psi,  \\ 
			\vspace{2mm}
			\ds\Psi_u   =\phi_3\Phi+\phi_4\Psi, \\
		\end{array}
	\end{equation} 
	\begin{equation} \label{eq:4.1-i}
		\begin{array}{l}
			\vspace{2mm}
			\Phi(u,v_0)=g_1(u), \\
			\vspace{2mm} 
			\Psi(u_0,v)=g_2(v).
		\end{array}
	\end{equation} 
	
	Note that this is an initial value problem for a canonical hyperbolic system of PDEs (see eq. (CHS) in \cite{Top}) and thus system \eqref{eq:4.1} determines (at least locally) functions  $\Phi(u,v)$ and $\Psi(u,v)$ 
	satisfying the initial conditions \eqref{eq:4.1-i}. {  Having in mind that $g_1$ and $g_2$ are exponentials, hence positive,  $\Phi$ and $\Psi$ are continuous and non-vanishing on a connected domain ${\mathcal D}_0 \subset {\mathcal D}$, we conclude that $\Phi >0$ and $\Psi >0$. }

	Now we  introduce the functions $ E=\Phi^2 $, $ G=\Psi^2 $,  and 
	\begin{equation} \label{eq:3.7}
		\ \gamma_1=-\frac{\Phi_v}{\Phi \Psi};  \qquad \gamma_2=-\frac{\Psi_u}{\Phi \Psi};
	\end{equation} 
	\begin{equation} \label{eq:3.7-a}
		\ds \beta_1= \frac{1}{\Phi \Psi} \left(\Psi (\ln |\mu|)_u + 2\Psi_u \right);  \qquad 
		\ds \beta_2= \frac{1}{\Phi \Psi} \left(\Phi (\ln|\mu|)_v + 2\Phi_v \right).
\end{equation} 

 First, we will show that the parameters $(u, v)$ are canonical. From \eqref{eq:g1&g2} and \eqref{eq:4.1-i} we have
	\begin{equation}\label{eqinitialc}
	g_1(u_0) = \Phi(u_0,v_0)=e^{-c_1}, \;\;\; g_2(v_0) = \Psi(u_0,v_0)=e^{-c_2}. 
	\end{equation}
Equations \eqref{eq:4.1} and \eqref{eq:4.1-i} give the following differential equation
	\begin{equation}\label{eqforPsi}
		\big( \Psi(u,v_0) \big)_u = \phi_3(u,v_0)g_1(u)  + \phi_4(u,v_0) \Psi(u,v_0).
	\end{equation}
On the other hand, from \eqref{eq:g1&g2} we have
\begin{equation*}
\ln g_1(u)= \ds\int_{u_0}^u\left(c\phi_3+\phi_4\right)(u,v_0)du-c_1,
\end{equation*}
{ and after differentiation we obtain} 
\begin{equation*}
{(g_1(u))_u} = \left(c\phi_3(u,v_0)+\phi_4(u,v_0)\right)g_1(u)
\end{equation*}
i.e.
\begin{equation}\label{eqforg1}
	\left( \frac{g_1(u)}{c}\right) _u = \phi_3(u,v_0) g_1(u)+\phi_4(u,v_0)\frac{g_1(u)}{c}.
\end{equation}
 Using that the functions $\Psi(u,v_0)$ and $\ds{\frac{g_1(u)}{c}}$ satisfy { one and the same  differential equation} (see \eqref{eqforPsi} and \eqref{eqforg1}) and have the same initial conditions
$$ \Psi(u_0,v_0)=e^{-c_2}; \quad { \frac{g_1(u_0)}{c}= \frac{e^{-c_1}}{c}= e^{-c_2},}$$
based on the uniqueness of the solution of the Cauchy problem, we conclude that 
$$ \Psi(u,v_0) = \frac{g_1(u)}{c},$$
 which imply that 
\begin{equation}\label{eqPhiPsi0}
\frac{\Phi(u,v_0)}{\Psi(u,v_0)} = \frac{g_1(u)}{\frac{g_1(u)}{c}} = c = e^{c_2-c_1}.
\end{equation}
Equations \eqref{eq:4.1} can also be written in the form
	\begin{equation*} 
	\begin{array}{l}
		\vspace{2mm}
		\ds(\ln\Phi)_v=\phi_1+\phi_2\frac{\Psi}{\Phi},  \\ 
		\vspace{2mm}
		\ds(\ln\Psi)_u   =\phi_3\frac{\Phi}{\Psi}+\phi_4. \\
	\end{array}
\end{equation*} 
 Hence, for the function $\varphi(u)$ defined by \eqref{eq:phi-psi} we have
\begin{equation*}
\begin{split}
 \varphi(u)& =\ds\Phi e^{-\ds\int_{v_0}^v(\ln\Phi)_v\, dv
 	-\int_{u_0}^u (\ln\Psi(u,v_0))_u\, du+c_1} \\ & = \Phi(u,v) \frac{\Phi(u,v_0)\Psi(u_0,v_0) }{\Phi(u,v)\Psi(u,v_0)}\, e^{\ds c_1},  
\end{split}
\end{equation*}
{ and keeping in mind \eqref{eqinitialc} and \eqref{eqPhiPsi0} we get $\varphi(u)= 1$}.
Similarly, we can also  obtain $\psi(v)=1$, which means that  $(u,v)$ are canonical principal parameters.

Further, using the introduced functions we can express  the right-hand side of the first equation of (\ref{eq:4.3}) as follows
\begin{equation*}
	\ds-\frac1{\Phi \Psi}
			\left( \left(\frac{\Phi_v}{\Psi}\right)_v+\left(\frac{\Psi_u}{\Phi}\right)_u \right)
	= \frac{ (\gamma_1)_v}{\sqrt G}+\frac{(\gamma_2)_u}{\sqrt E}-\Big((\gamma_1)^2+(\gamma_2)^2\Big).
\end{equation*}
Therefore, the first equation of (\ref{eq:4.3}) takes the form
$$
2\lambda\mu=\frac{(\gamma_2)_u}{\sqrt E}+\frac{ (\gamma_1)_v}{\sqrt G}-\Big((\gamma_1)^2+(\gamma_2)^2\Big) ,
$$
which coincides with the third equation of system (\ref{eq:BasicSystem}).

Similarly, the right-hand side of the second equation of (\ref{eq:4.3}) can be written as
\begin{equation*}
	\ds \frac{2}{\Phi\Psi} \left( \left(\frac{\Psi_u}{\Psi} \right)_v - \left(\frac{\Phi_v}{\Phi} \right)_u \right)
	= -\gamma_1\beta_1 + \gamma_2\beta_2 - \frac{(\beta_2)_u}{\sqrt E} + \frac{(\beta_1)_v}{\sqrt G},
\end{equation*}
and hence,  the second equation of (\ref{eq:4.3}) takes the form
$$
2\nu\mu  = \frac{(\beta_1)_v}{\sqrt G} - \frac{(\beta_2)_u}{\sqrt E} -\gamma_1\beta_1 + \gamma_2\beta_2, 
$$
which is exactly the last equation of system (\ref{eq:BasicSystem}).

Using equalities (\ref{eq:3.7}) and (\ref{eq:3.7-a}), we can easily calculate that
$$
2\mu \gamma_2 + \mu \beta_1 = -2 \mu \frac{\Psi_u}{\Phi \Psi} +\mu  \frac{1}{\Phi \Psi} \left(\Psi (\ln |\mu|)_u + 2\Psi_u \right) = \frac{\mu_u}{\sqrt{E}}, 
$$
which is the first equation of system (\ref{eq:BasicSystem}),
and 
$$
2\mu \gamma_1 + \mu \beta_2 = -2 \mu \frac{\Phi_v}{\Phi \Psi} +\mu  \frac{1}{\Phi \Psi} \left(\Phi (\ln |\mu|)_v + 2\Phi_v \right) = \frac{\mu_v}{\sqrt{G}}, 
$$
which is the second equation of system (\ref{eq:BasicSystem}).

To check the remaining two equations of system (\ref{eq:BasicSystem}) we use that $\Phi(u,v)$ and $\Psi(u,v)$  satisfy system  (\ref{eq:4.1}). After long but standard calculations, by use of \eqref{eq:functions}, \eqref{eq:4.1}, \eqref{eq:3.7} and \eqref{eq:3.7-a} we get that the fourth and fifth equations of 
\eqref{eq:BasicSystem} are also satisfied. Consequently, all conditions  in system \eqref{eq:BasicSystem} are fulfilled. 
Moreover, using  \eqref{eq:3.7} and \eqref{eq:3.7-a}  we can easily obtain that  the following inequalities are also valid
$$
 \Phi= \sqrt E  = \frac{\mu_u}{\mu(2\gamma_2+\beta_1)}>0 \ ;  \qquad\qquad
 \Psi= \sqrt G  =\frac{\mu_v}{\mu(2\gamma_1+\beta_2)}>0 \ ,
$$
where we have taken into account that $\mu\ne0$ and that the condition $\mu_u\mu_v\ne0$ is equivalent to $(2\gamma_2+\beta_1)(2\gamma_1+\beta_2)\ne0$ (based on the first two equations of system \eqref{eq:BasicSystem}). 

Finally, the end of the proof follows based on Theorem 4.1 in paper \cite{G-M-JMP-2012}. \end{proof}

In other words, Theorem \ref{eq:MainTheorem} states that each marginally trapped surface of general type in the Minkowski space $\R^4_1$ is determined up to a motion (a position in the space)  by three smooth functions 
$\nu(u,v)$, $\lambda(u,v)$, and  $\mu(u,v)$ ($\mu_u\mu_v\ne0$), satisfying a system of PDEs. Moreover, the parameters $(u,v)$ are the canonical principal parameters of the marginally trapped surface.

\section{Marginally trapped surfaces with parallel mean curvature vector field}

In this section we will consider marginally trapped surfaces in $\R^4_1$ with  parallel mean curvature vector field, i.e. the condition $DH =0$ holds identically. For such surfaces, the invariants $\beta_1$ and $\beta_2$ vanish, i.e. $\beta_1=\beta_2=0$ (see \cite{G-M-JMP-2012}). So, in the case of a marginally trapped surface with parallel mean curvature vector field, 
from the last equation of system (\ref{eq:BasicSystem}) it follows that $\nu=0$, since $\mu \neq  0$.

Now, the equations in system (\ref{eq:BasicSystem}) {  take the following simpler forms:}
\begin{equation}\label{eq:5.1mtswithparallelmcv}
	\begin{array}{l}
		\vspace{2mm}
		\ds 2\mu\gamma_2\sqrt{E} = \mu_u, \\
		\vspace{2mm}
		\ds	2\mu\gamma_1\sqrt{G}=\mu_v, \\
		\vspace{2mm} 
		\ds 2\lambda\gamma_2\sqrt{E}=\lambda_u, \\
		\vspace{2mm}
		\ds	2\lambda\gamma_1\sqrt{G}=\lambda_v, \\
		\vspace{2mm}
 	{ 	2\lambda\, \mu =
 		\displaystyle{\frac{(\gamma_2)_u}{\sqrt{E}}
 			+ \frac{(\gamma_1)_v}{\sqrt{G}} - \left((\gamma_1)^2 + (\gamma_2)^2\right)}. }
	\end{array}
\end{equation}
Using the { first four equalities in } \eqref{eq:5.1mtswithparallelmcv}, we obtain
	\begin{equation*}
					\left( \frac{\lambda}{\mu} \right)_u =0, \quad
			\left( \frac{\lambda}{\mu} \right)_v =0.
	\end{equation*}
Hence, the functions $\mu(u,v)$ and $\lambda(u,v)$ are proportional, i.e. 
\begin{equation}\label{eq:mulambda}
	\lambda(u,v) = c_3\, \mu(u,v),
\end{equation}
where $c_3 = \const$.

{ In the case $\nu=0$ and  $4\nu^2+4\lambda^2-1\ne0$, i.e. $\lambda\ne\pm \frac{1}{2}$ { (or, equivalently, $4c^2_3\mu^2\ne1$)}, we can use formulas \eqref{eq:functions} and  obtain } the following expressions for the functions $\phi_i$, $i=1,2,3,4$
\begin{equation*}
	\phi_1 = -(\ln\sqrt{|\mu|})_v, \;\; \phi_2 =\phi_3 = 0, \;\; \phi_4  =-(\ln\sqrt{|\mu|})_u.
\end{equation*}
If we substitute these expressions in \eqref{eq:phi-psi}, we get that the function $\varphi(u)$ is expressed as
\begin{equation*}
	\begin{split}
		\varphi(u) & =\sqrt E\, e^{\ds-\int_{v_0}^v\left(\phi_1+\phi_2\frac{\sqrt G}{\sqrt E}\right)dv
			-\int_{u_0}^u\left(\phi_3\frac{\sqrt E}{\sqrt G}+\phi_4\right)(u,v_0)du+c_1}  \\ & =
		\sqrt E\, e^{\ds \ln{\sqrt{|\mu(u,v)|}} - \ln{\sqrt{|\mu(u_0,v_0)|}} + c_1}.
\end{split} \end{equation*}
Choosing $c_1=\ln{\sqrt{|\mu(u_0,v_0)|}}$, we  obtain
\begin{equation}\label{E:phi}
 \varphi(u)  =\sqrt E\,\sqrt{|\mu|}.
\end{equation}
Similarly, choosing $c_2=\ln{\sqrt{|\mu(u_0,v_0)|}}$, we obtain that the  function $\psi(v)$ has the following form
\begin{equation}\label{E:psi}
\psi(v)  =\sqrt G\,\sqrt{|\mu|}.
\end{equation}

Hence, applying Definition \ref{D:def-can}  of canonical parameters for marginally trapped surfaces with parallel mean curvature vector field, we obtain that $(u,v)$ are canonical parameters if and only if 
  $$\sqrt E\,\sqrt{|\mu|} = 1; \quad \sqrt G\,\sqrt{|\mu|} = 1,$$
	or equivalently, 
 $$E(u,v)=\frac{1}{|\mu(u,v)|}; \quad G(u,v)=\frac{1}{|\mu(u,v)|}.$$

\vskip 1mm
{ Now, we will show that  we can directly obtain functions $\varphi(u)$ and $\psi(v)$ satisfying  \eqref{E:phi} and \eqref{E:psi}, respectively,  for all marginally trapped surfaces with parallel mean curvature vector field without using $\phi_i$, $i=1,2,3,4$ and the restrictions above.
Indeed, keeping in mind that $\gamma_1=\ds{-\frac{1}{\sqrt G} (\ln \sqrt E)_v}$ and  $\gamma_2=\ds{-\frac{1}{\sqrt E} (\ln \sqrt G)_u}$, from the first two equalities in \eqref{eq:5.1mtswithparallelmcv} we obtain that
$$(\ln|\mu| E)_v = 0, \quad (\ln|\mu| G)_u = 0,$$
which imply that  
$|\mu| E$ does not depend on $v$, and $|\mu| G$ does not depend on $u$.
Hence, there exist functions $\varphi(u) >0 $ and $ \psi(v) >0$, such that
$$\varphi(u) = \sqrt{|\mu|} \sqrt{E}; \qquad \psi(v) = \sqrt{|\mu|} \sqrt{G}.$$
Then, under the following change of the parameters:
\begin{equation*}  
\begin{array}{l}
\vspace{2mm}
\overline{u} = \ds{\int_{u_0}^u \varphi(u)\, du} + \overline{u}_0, \quad \overline{u}_0 = const\\
[2mm]
\overline{v} = \ds{\int_{v_0}^v  \psi(v)\, dv + \overline{v}_0}, \quad \overline{v}_0 = const
\end{array}
\end{equation*}
we obtain
\begin{equation} \label{E:Eq-can}
\overline{E} = \ds{\frac{1}{|\mu|}}; \qquad \overline{F} = 0; \qquad \overline{G} = \ds{\frac{1}{|\mu|}}.
\end{equation}
Consequently, for all marginally trapped surfaces with parallel mean curvature vector field (including the cases $\mu_u \mu_v = 0$ and $4\nu^2+4\lambda^2-1 = 0$) we can introduce parameters satisfying the equalities \eqref{E:Eq-can}. Such parameters we call  \textit{canonical principal parameters} of the surface.

\vskip 2mm
Finally, in the class of marginally trapped surfaces with parallel mean curvature vector field, the system \eqref{eq:5.1mtswithparallelmcv} reduces to a single equation, which is exactly the Gauss equation. 
Indeed, in terms of canonical parameters we have:
$$E=G= \ds{\frac{1}{|\mu|}}; \quad \gamma_1 = \left(\sqrt{|\mu|}\right)_v; \quad \gamma_2 = \left(\sqrt{|\mu|}\right)_u; \quad \lambda = c_3 \mu,$$
and from the last equality of \eqref{eq:5.1mtswithparallelmcv} we get the following differential equation:
$$\Delta \ln |\mu| = 4c_3\,|\mu|, \;\; c_3\in\mathbb{R}.$$

Now we can formulate the fundamental existence and uniqueness theorem for the class of marginally trapped surfaces
with parallel mean curvature vector field in terms of canonical parameters.

	\begin{thm}
		\label{eq:MainTheoremparallel}
		Let $\mu(u,v)\ne0$ be a smooth function  defined in a domain ${\mathcal D} \subset \R^2$  and satisfying the equation
		\begin{equation} \label{eq:4.3p} 
				\Delta \ln |\mu| = 4c_3\,|\mu|, \;\; c_3\in\mathbb{R},
		\end{equation}
		where  $\Delta=\partial_{uu}+\partial_{vv}$ is the Laplacian. Then, there exists a unique (up to a position in $\mathbb R^4_1$) marginally trapped surface with parallel mean curvature vector field parametrized by canonical principal parameters $(u,v)$ with geometric function $\mu(u,v)$.
	\end{thm}

\begin{proof}
Let $\mu(u,v)\ne0$, $(u,v) \in {\mathcal D} \subset \R^2$ be a smooth function. We define the functions 
$$E=G= \ds{\frac{1}{|\mu|}}; \quad \gamma_1 = \left(\sqrt{|\mu|}\right)_v; \quad \gamma_2 = \left(\sqrt{|\mu|}\right)_u; \quad \lambda = c_3 \mu,$$
and consider the following
system of partial differential equations 
\begin{equation} \label{E:Eq-system-1}
\begin{array}{ll}
\vspace{2mm} x_u = \sqrt{E} \left(\gamma_1\, y + \,n_1\right),
& \quad x_v = \sqrt{G}\left(- \gamma_2\, y + \lambda\, n_1 + \mu\, n_2\right), \\
\vspace{2mm} y_u = \sqrt{E} \left(- \gamma_1\, x + \lambda\, n_1 + \mu\, n_2\right),  &
\quad y_v = \sqrt{G}\left(\gamma_2\, x +\, n_1 \right), \\
\vspace{2mm} (n_1)_u =  \sqrt{E} \mu\, y,    &
\quad (n_1)_v =  \sqrt{G} \mu\, x, \\
\vspace{2mm} (n_2)_u = \sqrt{E}\left( x +  \lambda\, y \right), &
\quad (n_2)_v = \sqrt{G}\left(\lambda\, x +  \,y \right), 
\end{array}
\end{equation}
for the unknown vector
functions $x = x(u,v), \, y = y(u,v), \,n_1 = n_1(u,v), \,n_2 = n_2(u,v)$ in $\R^4_1$. System \eqref{E:Eq-system-1} can be written in the form.
\begin{equation} \label{E:Eq-system-2}
\begin{array}{ll}
\vspace{2mm} x_u = \ds{\frac{|\mu|_v}{2|\mu|}\, y + \frac{1}{\sqrt{|\mu|}}\,n_1},
& \quad x_v = \ds{-\frac{|\mu|_u}{2|\mu|}\, y + c_3\frac{\mu}{\sqrt{|\mu|}}\, n_1 +\frac{\mu}{\sqrt{|\mu|}}\, n_2}, \\
\vspace{2mm} y_u = \ds{-\frac{|\mu|_v}{2|\mu|}\, x + c_3\frac{\mu}{\sqrt{|\mu|}}\, n_1 + \frac{\mu}{\sqrt{|\mu|}}\, n_2 }, &
\quad y_v = \ds{\frac{|\mu|_u}{2|\mu|}\, x + \frac{1}{\sqrt{|\mu|}}\, n_1 },\\
\vspace{2mm} (n_1)_u = \ds{ \frac{\mu}{\sqrt{|\mu|}}\, y },  &
\quad (n_1)_v =  \ds{\frac{\mu}{\sqrt{|\mu|}}\, x },\\
\vspace{2mm} (n_2)_u = \ds{\frac{1}{\sqrt{|\mu|}}\,x +  c_3\frac{\mu}{\sqrt{|\mu|}}\, y }, &
\quad (n_2)_v = \ds{c_3\frac{\mu}{\sqrt{|\mu|}}\, x + \frac{1}{\sqrt{|\mu|}} \,y }.
\end{array}
\end{equation}

It can easily be checked that the integrability conditions of  \eqref{E:Eq-system-2}  are fulfilled, based on \eqref{eq:4.3p}.  
 Hence, if we are given an orthonormal quadruple  $\{x_0, \, y_0, \, (n_1)_0,\, (n_2)_0\}$ at a point $p_0 \in \R^4_1$, such that $x_0$, $y_0$ are
unit spacelike vectors, $\langle x_0, y_0 \rangle = 0$, $(n_1)_0,
(n_2)_0$ are lightlike vectors, and $\langle (n_1)_0, (n_2)_0
\rangle = -1$, then 
there exist a subset
$\mathcal{D}_1 \subset \mathcal{D}$ and unique vector functions $x
= x(u,v), \, y = y(u,v), \,n_1 = n_1(u,v), \,n_2 = n_2(u,v), \,\, (u,v)
\in \mathcal{D}_1$, which satisfy system \eqref{E:Eq-system-2} and the conditions
$$x(u_0,v_0) = x_0, \quad y(u_0,v_0) = y_0, \quad n_1(u_0,v_0) = (n_1)_0, \quad
n_2(u_0,v_0) = (n_2)_0.$$

Now, let us consider the  system
\begin{equation}\label{E:Eq-system-3}
\begin{array}{lll}
\vspace{2mm}
z_u =\ds{\frac{1}{\sqrt{|\mu|}}\, x}, \\
\vspace{2mm} z_v =\ds{ \frac{1}{\sqrt{|\mu|}}\, y}, 
\end{array}
\end{equation}
for the vector function $z(u,v)$. Using \eqref{E:Eq-system-2} one can easily check that the integrability
conditions $z_{uv} = z_{vu}$ of system \eqref{E:Eq-system-3}
 are fulfilled. Hence,  there exist a subset $\mathcal{D}_0 \subset \mathcal{D}_1$ and
a unique vector function $z = z(u,v)$, defined for $(u,v) \in
\mathcal{D}_0$ and satisfying $z(u_0, v_0) = p_0$.

Using \eqref{E:Eq-system-2} we obtain that  $M^2: z = z(u,v), \,\, (u,v) \in \mathcal{D}_0$ is a marginally trapped surface with parallel mean curvature vector field. And moreover, the parameters  $(u,v)$  are canonical principal parameters.

\end{proof}
}

{
\begin{remark}
Canonical parameters for spacelike surfaces with parallel normalized mean curvature vector field in the Minkowski 4-space $\R^4_1$ were introduced in \cite{G-M-Fil} as special isothermal parameters that satisfy the conditions 
$$E(u,v)=G(u,v)=\frac{1}{|\mu(u,v)|}; \quad F(u,v)=0.$$
Obviously, for marginally trapped surfaces with parallel mean curvature vector field  we have the same expressions for the coefficients of the first fundamental form.  
\end{remark}
}
\begin{remark}
A classification of marginally trapped surfaces with parallel mean curvature vector field is given by B.-Y. Chen and J. Van der Veken in \cite{Chen-Veken-2}. According to their classification, there exist six types of marginally trapped surfaces with parallel mean curvature vector field which are listed in Theorem 4.1 in  \cite{Chen-Veken-2}.  Below, we will discuss types (5) and (6) which lie in $\mathbb{S}^3_1(1)$ and $\mathbb{H}^3(-1)$, respectively.
\end{remark}

{

\begin{example}
 Let us consider the following marginally trapped surface with parallel mean curvature vector field (see Example 5.2 in \cite{Chen-Veken-2}):
	\begin{equation*}
	\begin{split}
		{M}^2:\  	z(u,v)=\frac{1}{\sqrt2}\Bigl( & u\sinh(\sqrt2 v),\ \
	\sqrt2\sin(\sqrt2 u)-u\cos(\sqrt2 u), \\ & 
	\sqrt2\cos(\sqrt2 u)+u\sin(\sqrt2 u),\ \
	u\cosh(\sqrt2 v)\Bigr), \;\; u>0.
	\end{split}
	\end{equation*}
	
This surface lies in $\mathbb{S}^3_1(1)$. Since $\langle z,z\rangle=1$, we can consider the following normal vector fields
\begin{equation*}
\begin{split}
{ n_1 } = & z, \\
{ n_2 }= &\frac{1}{\sqrt2\,u}\Bigl((u^{2}+1)\sinh(\sqrt2 v),\ (1-u^{2})\cos(\sqrt2u)
+\sqrt2\,u\sin(\sqrt2u), \\
 & \qquad\quad (u^{2}-1)\sin(\sqrt2u)+\sqrt2\,u\cos(\sqrt2u),\
(u^{2}+1)\cosh(\sqrt2 v)\Bigr),
\end{split}
\end{equation*}
{ which satisfy  $\langle n_1, n_1\rangle=1$, $\langle n_2, n_2\rangle=-1$,
$\langle n_1,n_2\rangle=0$. } The coefficients of the { first and the } second fundamental form are given, { respectively,}  by
\begin{equation*}
\begin{split}
	 E & =G=u^{2},\qquad F=0, \\
 L & =N=0, \; \qquad M=-2\ne0,
\end{split}
\end{equation*}
which means that the parameters $(u,v)$ are isothermal, but not principal, since $M\ne0$. If we denote
$x =\ds{\frac{z_u}{\sqrt{E}}}$, $y = \ds{\frac{z_v}{\sqrt{G}}}$, then for the second fundamental tensor $\sigma$,  with respect to the 
frame field  { $\{x,y,n_1,n_2 \}$}, we have 
\begin{equation*}
	\begin{array}{l}
		\vspace{2mm}
		\sigma(x,x) = \ds{-n_1+\frac{u^{2}-1}{u^{2}} \,n_2}; \\
		\vspace{2mm}
		\sigma(x,y) = 0;  \\
		\vspace{2mm}
		\sigma(y,y) = \ds{-n_1+\frac{u^{2}+1}{u^{2}} \,n_2,}
	\end{array}
\end{equation*}
and hence the mean curvature vector field is $H= \ds{\frac{1}{2}\mathrm{tr}\sigma} = -n_1+n_2$, which confirms that the surfaces is marginally trapped, i.e. $H\ne0$ and $\langle H, H \rangle = 0$.  

To obtain  principal parameters to which we can apply our canonical parameter theory, we { consider  the following  change of the parameters: }
\begin{equation}\label{eqex1rot}
	\bar u=\frac{u+v}{\sqrt2},\quad \bar v=\frac{u-v}{\sqrt2},
\end{equation}
i.e.
\begin{equation*}
 u=\frac{\bar u+ \bar v}{\sqrt2},\quad  v=\frac{\bar u- \bar v}{\sqrt2}.
\end{equation*}
 With respect to the parameters $(\bar u, \bar v)$ the surface ${M}^2$ { has the following parametrization:}
 \begin{equation*}
\begin{split}
{M}^2:\ z(\bar u, \bar v)= \Bigl( & \tfrac{\bar u + \bar v}{2}\sinh (\bar u - \bar v),\ \ \sin (\bar u + \bar v)-\tfrac{\bar u + \bar v}{2}\cos (\bar u + \bar v),\ \\ &
 \cos (\bar u + \bar v)+\tfrac{\bar u + \bar v}{2}\sin (\bar u + \bar v),\ \ \tfrac{\bar u + \bar v}{2}\cosh (\bar u - \bar v)\Bigr),\;\; \bar u + \bar v>0,	 
\end{split}
 \end{equation*}
and the coefficients of the first and the second fundamental form become
 \begin{equation}\label{eqex1barEGF}
 	\begin{split}
 		\bar E & = \bar G=\frac{(\bar u + \bar v)^{2}}{2},\qquad \bar F=0, \\
 		\bar L & = - \bar N=2, \; \qquad \qquad \bar M=0,
 	\end{split}
 \end{equation}
which means that the parameters $(\bar u, \bar v)$ are principal. 
{ The geometric frame field $\{X,Y,N_1,N_2 \}$ in the sense of \cite{G-M-JMP-2012} (where $X$ and $Y$ are  principal directions, $N_1=H$, and $N_2$ is the other lightlike normal vector field) is given by the following formulas:} 
\begin{equation*} 
	\begin{array}{ll}
		\vspace{2mm}
		X =\ds{\frac{z_{\bar u}}{\sqrt{\bar E}}}; & \quad N_1 = -n_1+n_2;\\
		\vspace{2mm}
		Y = \ds{\frac{z_{\bar v}}{ \sqrt{\bar G}}}; & \quad N_2 = \frac{1}{2} (n_1   + n_2).
	\end{array}
\end{equation*} 
The second fundamental tensor with respect to the geometric frame field $\{X,Y,N_1,N_2 \}$ { is given by }
\begin{equation} \label{Eq:geom-field}
	\begin{array}{l}
		\vspace{2mm}
		\sigma(X,X) =  N_1; \\
		\vspace{2mm}
		\sigma(X,Y) = -\ds{\frac{1}{(\bar{u}+\bar{v})^2}}\, N_1 - \ds\frac{2}{(\bar{u}+\bar{v})^2}\, N_2;  \\
		\vspace{2mm}
		\sigma(Y,Y) = N_1.
	\end{array}
\end{equation}
{ Now, comparing equalities \ref{Eq:geom-field} with formulas (2) in \cite{G-M-JMP-2012}, we obtain }
\begin{equation}\label{eqex1numul}
 \nu=0,\qquad \mu=-\frac{2}{(\bar{u}+\bar{v})^2},
 \qquad \lambda=-\frac{1}{(\bar{u}+\bar{v})^2},
\end{equation}
which imply that $\lambda=c_3 \mu$, where $c_3=\ds{\frac12}$.  Using   \eqref{eqex1barEGF} and \eqref{eqex1numul}, we   get  $\bar E=\bar G=1/|\mu|$, so we conclude that the parameters $(\bar{u},\bar{v})$ are canonical. 
{ In this example, the } surface is of general type at
every point {  where } $\bar u +\bar v\neq\sqrt2$.
{ Moreover, the following equation holds:
$$\Delta \ln |\mu| = 4c_3\,|\mu|,$$
where $c_3=\ds{\frac{1}{2}}$. }

\end{example}

\vskip 2mm

\begin{example}
 In Example 5.3 of \cite{Chen-Veken-2}, a marginally
 	trapped surface with parallel mean curvature vector field in {  $\mathbb{H}^3(-1)$ is  given by the following parametrization:}
 \begin{equation*}
 \begin{split}
 z(u,v)= \tfrac{1}{\sqrt6} \Bigl( & {\sqrt2}\cos(\sqrt3 v)\cosh u,  \sqrt{7+\cosh 2u}\sinh\Bigl(\sqrt3\,u-\operatorname{artanh}\bigl(\tfrac{\sqrt3}{2}\tanh u\bigr) \Bigr), \\ &
  {\sqrt2}\sin(\sqrt3 v)\cosh u, \sqrt{7+\cosh 2u}\cosh\Bigl(\sqrt3\,u-\operatorname{artanh}\bigl(\tfrac{\sqrt3}{2}\tanh u\bigr) \Bigr)\Bigr).
 \end{split}
 \end{equation*}

For this surface  $\langle z,z\rangle=-1$ holds and we can consider the orthonormal normal {  frame field $\{n_1,n_2\}$ given by the following equalities }
\begin{equation*}
	\begin{split}
		{ n_1 } = &\Bigl(
		-\frac{\cosh^{2}u+1}{\sqrt{3}\,\cosh u}\,\cos(\sqrt{3}\,v),\
		\sinh u\,\cosh(\sqrt{3}\,u)-\frac{\cosh 2u}{\sqrt{3}\,\cosh u}\,\sinh(\sqrt{3}\,u),\\  & \;\;
		-\frac{\cosh^{2}u+1}{\sqrt{3}\,\cosh u}\,\sin(\sqrt{3}\,v),\
		\sinh u\,\sinh(\sqrt{3}\,u)-\frac{\cosh 2u}{\sqrt{3}\,\cosh u}\,\cosh(\sqrt{3}\,u)
		\Bigr), \\
		{ n_2} = & z.
	\end{split}
\end{equation*}
The coefficients of the first and the second fundamental form are given by
\begin{equation*}
\begin{split}
  E=& G=\cosh^{2}u,\qquad F=0, \\
  L=& N=0, \qquad M=2.
\end{split}
\end{equation*}
After { changing the parameters $(u,v)$ with $(\bar u,\bar v)$ } in the same way as in the previous example, i.e. using the { equalities} \eqref{eqex1rot}, we get
 \[
 \bar E=\bar G=\ds{\cosh^{2}\tfrac{\bar u + \bar v}{\sqrt2},\quad \bar F=0,\quad \bar M=0}.
 \]
 { Hence,} $(\bar u,\bar v)$ are principal parameters. By calculating the second fundamental tensor $\sigma$ with respect to the geometric frame field $\{X,Y,N_1,N_2 \}$, where
\begin{equation*} 
	\begin{array}{ll}
		\vspace{2mm}
		X =\ds{\frac{z_{\bar u}}{\sqrt{\bar E}}}; & \quad N_1 = H = n_1+n_2;\\
		\vspace{2mm}
		Y = \ds{\frac{z_{\bar v}}{ \sqrt{\bar G}}}; & \quad N_2 = \frac12(-n_1+n_2),
	\end{array}
\end{equation*} 
 we obtain
 \begin{equation*}  \nu=0, \quad
 \mu=\ds{\operatorname{sech}^{2}\tfrac{\bar u + \bar v}{\sqrt2}},\quad
 \lambda=\ds{-\tfrac12\operatorname{sech}^{2}\tfrac{\bar u + \bar v}{\sqrt2}},\quad
 c_3=\ds{-\tfrac12}.
\end{equation*} 
 
We see that $\bar E=\bar G=1/|\mu|$ also applies in this example, i.e. $(\bar{u},\bar{v})$ are canonical parameters. 
The surface is of general type at
 every point {  where} $\bar u +\bar v\neq 0$.
{ The  equation
$$\Delta \ln |\mu| = 4c_3\,|\mu|,$$
 holds, with $c_3=\ds{-\frac{1}{2}}$. }

\end{example}

\vskip 2mm

\begin{remark} 
{ Equation } \eqref{eq:4.3p} is equivalent to equations (5.4) and (5.6) of \cite{Chen-Veken-2}. 
{ The parameters $(u,v)$ used by Chen and Van der Veken in both  Examples 5.2 and 5.3 in \cite{Chen-Veken-2} are isothermal, but they are not principal,
so by Definition \ref{D:def-can} they are not canonical. }
After a rotation of the coordinate system given by \eqref{eqex1rot}, the new parameters become canonical. 
\end{remark}

\begin{remark}
{ In the case } $c_3=0$, from \eqref{eq:4.3p} it follows { that } $\lambda=0$, the function $\ln |\mu|$ is harmonic and equalities \eqref{eq:K_kappa} imply that both 
the Gauss curvature $K$ and the  normal curvature $\varkappa$ are zero. Such a type of { marginally trapped  surfaces } is presented in Theorem 4.1 of \cite{Chen-Veken-2}, cases (2) and (4). 
\end{remark}

}

\begin{remark}
	Since  $\nu=0$ for marginally trapped surfaces with parallel mean curvature vector field, from the equation $\varkappa=-2\mu\nu$, we get $\varkappa=0$, which means that all marginally trapped surfaces with parallel mean curvature vector field are surfaces with flat normal connection (see also Proposition 3.2 in \cite{G-M-JMP-2012}).
\end{remark}

\section{Examples }

In this section, we will show how to find the canonical principal parameters on a special family of marginally trapped surfaces -- the marginally trapped meridian surfaces of parabolic type. { Unlike the previous section, these are surfaces that do not have a parallel mean curvature vector.}

In \cite{GM4}, a family of surfaces lying on a standard rotational hypersurface in the four-dimensional Euclidean space
$\R^4$ was constructed.  
They are called \emph{meridian surfaces} because they are one-parameter systems of meridians of the rotational hypersurface. Using a similar idea,  special families of two-dimensional spacelike surfaces lying
on rotational hypersurfaces in $\R^4_1$ were constructed in papers \cite{G-M-JMP-2012,G-M-IJGMMP-2013}. Depending on the causal character of the axis of the rotational hypersurface, as well as the type of the spheres in
a 3-dimensional Minkowski subspace and the causal character of the spherical curves, we
distinguish different types of spacelike or timelike  meridian surfaces  in $\R^4_1$ -- of elliptic type, hyperbolic type, or parabolic type.

The most interesting rotation is the rotation about a lightlike axis, so here we will consider two-dimensional surfaces lying on a rotational hypersurface with lightlike axis -- these are  the so-called meridian surfaces of  parabolic type.

\vskip 2mm
For the construction of meridian surfaces of parabolic type it is convenient to use the  pseudo-orthonormal base
 $\{e_1,e_2, \xi_1, \xi_2 \}$  of $\R^4_1$, where $e_1^2 = e_2^2 = e_3^2 = 1$, $e_4^2 = -1$ and 
$ \ds{\xi_1= 	\frac{e_3 + e_4}{\sqrt{2}}}$, $\ds{\xi_2= \frac{ - e_3 + 		e_4}{\sqrt{2}}}$ (see \cite{G-M-IJGMMP-2013}). It is clear that $\langle\xi_1, \xi_1 \rangle =0$, $\langle
\xi_2, \xi_2 \rangle =0$, $\langle \xi_1, \xi_2
\rangle = -1$. A rotational hypersurface with lightlike axis  $\xi_2$ can
be parametrized by
$$\mathcal{M}''': Z(u,w^1,w^2) =  f(u)\, w^1 \cos w^2 \,e_1 +  f(u)\, w^1 \sin w^2 \,e_2+ \left(f(u) \frac{(w^1)^2}{2} + g(u)\right) \xi_1 + f(u) \,\xi_2,$$
where $f = f(u), \,\, g = g(u)$ are smooth functions, defined in an interval $I \subset \R$, such that $- f'(u)g'(u) >0$, $f(u)>0$,
$ u \in I$. Assuming that $w^1 = w^1(v)$, $w^2=w^2(v)$, $v \in J \subset \R$,
and $(\dot{w}^1)^2 + (\dot{w}^2)^2 \neq 0$, we consider the one-parameter system of meridians
of the rotational hypersurface  with lightlike
axis defined by
\begin{equation*} 
	\mathcal{M}'''_m: z(u,v) = Z(u,w^1(v),w^2(v)),
\end{equation*}
where $u \in I$, $v \in J$. This surface is called a \textit{meridian surface 	of parabolic type}. 

Without loss of generality, we may assume that $w^1 = \omega(v)$, $
w^2=v$, so the parametrization of the  surface $\mathcal{M}'''_m$ takes the following form 
\begin{equation*} 
	\mathcal{M}'''_m: z(u,v) = f(u)\, \omega(v) \cos v \,e_1 + f(u)\,
	\omega(v) \sin v \,e_2+ \left(f(u) \frac{(\omega(v))^2}{2} +
	g(u)\right)\xi_1 + f(u) \,\xi_2.
\end{equation*}

In \cite{G-M-IJGMMP-2013}, it is shown that each meridian surface of parabolic type is generated by a plane meridian curve $m$ with curvature  
$\varkappa_m(u) = \ds{ \frac{f' g'' - g' f''}{(-2 f' g')^{\frac{3}{2}}}}$, and a curve $\overline{c}$ with curvature  $\overline{\kappa}(v) = \ds{\frac{\omega \ddot{\omega} - 2 \dot{\omega}^2 - \omega^2 }{(\dot{\omega} ^2 + \omega^2)^{\frac{3}{2}}}}$ lying on the paraboloid  $\mathcal{P}^2$ in $\R^4_1$,
defined by
$$\mathcal{P}^2: z(w^1,w^2) =  w^1 \cos w^2 \,e_1 +   w^1 \sin w^2 \,e_2+ \frac{(w^1)^2}{2} \, \xi_1 + \xi_2.$$
In the case  $f(u)=u$, $g(u)=au+b$, i.e. $\varkappa_m(u) =0$, the meridian curve is a straight line and $\mathcal{M}'''_m$ is a developable ruled  surface in $\R^4_1$ which is a cone \cite{G-M-IJGMMP-2013}. In the case $\varkappa_m(u) \neq 0$, the meridian surface of parabolic type $\mathcal{M}'''_m$
	is marginally trapped if and only if
	$\overline{\kappa}(v) = a = \const, \; a \neq 0$, and the
	meridian curve is defined by
	\begin{equation}  \notag
		\begin{array}{l}
			\vspace{2mm}
			f(u) = u>0;\\
			\vspace{2mm}
			g(u) = \ds{\frac{\pm 1}{2a^3}\left(\frac{a^2 u^2 \mp 2 auc}{c \mp au} -2c \ln |c\mp au| + b \right)},
		\end{array}
	\end{equation}
	where $b$ and $c$ are constants, $c \neq 0$ (Theorem 3.2,  \cite{G-M-IJGMMP-2013}).
	
	\vskip 2mm
	Now, we will show how to find the canonical principal parameters of the marginally trapped meridian surface of parabolic type for the following choice of the constants: $a= -1$, $b=0$,  and $c=1$. In a similar way, one can find the canonical principal parameters for arbitrary constants. 

So, let us consider the functions
\begin{equation*}  
		\begin{array}{l}
			\vspace{2mm}
			f(u) =  u >0;\\
			\vspace{2mm}
			g(u) = \ds{-\frac{1}{2}\left(\frac{u^2 +2 u}{1+u} -2 \ln (u+1) \right)}.
		\end{array}
	\end{equation*}
Then, the curvature $\varkappa_m(u)$ has the following form
$$\varkappa_m(u) = -\ds{\frac{1}{u^2}}.$$
Since the curvature $\overline{\kappa}(v) = a = -1$, the function  $\omega(v)$ must satisfy the differential equation:
$$\omega \ddot{\omega} - 2 \dot{\omega}^2 - \omega^2 = - (\dot{\omega} ^2 + \omega^2)^{\frac{3}{2}},$$
whose solution is given by  $\omega(v)=2\cos(v+p)$, $p\in\mathbb{R}$. For simplicity, we choose $p=0$, i.e. $\omega(v)=2\cos v$.

So, we consider the surface $\mathcal{M}$ with the following parametrization
$$\mathcal{M}: z(u,v) = 2 u \cos^2 v \,e_1 +  u \sin 2v \,e_2+ \left(2u \cos^2 v  + \ln (u+1) - \frac{u^2 +2 u}{2(u+1)} \right) \xi_1 + u \,\xi_2.$$

It can easily be calculated that the coefficients of the first fundamental form are given by  
\begin{equation*}
	E = \frac{u^2}{(u+1)^2}; \qquad F =0; \qquad G =4u^2,
\end{equation*}
and the coefficients of the second fundamental form (see the explicit formulas in \cite{G-M-Med} and \cite{G-M-IJGMMP-2013}) are 
$$L=0; \qquad M=-\frac{2}{u(u+1)}; \qquad N=0. $$

Since $M\ne0$, the parameters $(u,v)$ are not principal. To find principal parametrization of the surface,  we need to find parameters $(\bar u, \bar v)$ for which $\bar F=0$ and $\bar M=0$ (see \cite{G-M-JMP-2012}). Let us consider the following change of the parametrization
	\begin{equation*}
	\begin{split}
		u & = e^{ \bar{u}-\bar{v}} -1,
		\\	
		v & = \frac{ \bar{u} + \bar{v}}{2},
\end{split}
\end{equation*}
i.e.
	\begin{equation*}
		\begin{split}
			\bar{u} & = v +\ln\sqrt{u+1},
			\\	
			\bar{v} & = v -\ln\sqrt{u+1}.
		\end{split}
\end{equation*}
Direct computations show  that the conditions $\bar{F}=\bar{M}=0$ are valid for the new  parameters $(\bar u, \bar v)$, which means that the parameters  are principal. Taking into account that $f(u) = u>0$, the expression $u = e^{ \bar{u}-\bar{v}} -1$ implies that $\bar u > \bar v$.

With respect to the principal parameters, the coefficients of the first fundamental form are expressed as
\begin{equation*}
	\bar E = 2(e^{\bar{u}-\bar{v}} -1)^2; \qquad \bar F =0; \qquad \bar G = 2(e^{\bar{u}-\bar{v}} -1)^2.
\end{equation*}

According to \cite{G-M-IJGMMP-2013}, the mean curvature vector field $H$ of the meridian surface in our case is given by  
$$H = -\frac{1}{2u} (n_1  -  n_2),$$
where 
\begin{equation*} 
	\begin{array}{l}
		\vspace{2mm}
		n_1 =  \cos 2v\,e_1 + \sin 2v \,e_2+ 2 \cos^2 v \, \xi_1;\\
		\vspace{2mm}
		n_2 = \ds{\frac{u+1}{u} \left( 2 \cos^2 v \,e_1 + \sin 2v \,e_2 + \left(2 \cos^2 v + \frac{u^2}{2(u+1)^2}\right)\, \xi_1 +
			\xi_2\right)}.
	\end{array}
\end{equation*}
 Note that $\langle n_1,n_1\rangle = 1$, $\langle n_1,n_2\rangle = 0$, $\langle n_2,n_2\rangle = -1$. 

 The geometric frame field $\{X,Y,N_1,N_2 \}$ in the sense of \cite{G-M-JMP-2012} (where $X$ and $Y$ are  principal directions, $H = N_1$) is given by
\begin{equation*} 
\begin{array}{ll}
		\vspace{2mm}
X =\ds{\frac{z_{\bar u}}{\sqrt{\bar E}}}; & \quad N_1 = \ds{-\frac{1}{2u} (n_1  - n_2)};\\
		\vspace{2mm}
Y = \ds{\frac{z_{\bar v}}{ \sqrt{\bar G}}}; & \quad N_2 = u (n_1  + n_2).
	\end{array}
\end{equation*} 
The normal vector fields $N_1$, $N_2$ are a pseudo-orthonormal lightlike pair. 

By direct computations it can be checked that with respect to the geometric frame field $\{X,Y,N_1,N_2 \}$  and the principal parameters $(\bar u, \bar v)$ we have the following expressions
\begin{equation} \notag
	\begin{array}{l}
		\vspace{2mm}
		\sigma(X,X) =  N_1; \\
		\vspace{2mm}
		\sigma(X,Y) = \ds{\frac{e^{ \bar{u}-\bar{v}}}{e^{ \bar{u}-\bar{v}} -1}}\, \,N_1 + \ds{\frac{1}{2 (e^{ \bar{u}-\bar{v}} -1)^3}}\,\, N_2;  \\
		\vspace{2mm}
		\sigma(Y,Y) = N_1,
	\end{array}
\end{equation}
which allow us to derive the invariants $\nu$, $\lambda$, and  $\mu$ in terms of the parameters $(\bar u, \bar v)$
\begin{equation}\notag
	\nu = 0; \qquad \lambda = \ds{\frac{e^{ \bar{u}-\bar{v}} }{e^{ \bar{u}-\bar{v}} -1}}; \qquad \mu =\ds{ \frac{1}{2 (e^{ \bar{u}-\bar{v}} -1)^3}}.
\end{equation}
Taking into account the last equations, from \eqref{eq:functions} we can easily calculate that the functions $\phi_i$, $i=1,2,3,4$, take the following forms
 \begin{equation*}
		\phi_1 =- \phi_4=- \frac{ e^{ \bar{u}-\bar{v}} (e^{ \bar{u}-\bar{v}} + 3)}{2(e^{ 2(\bar{u}-\bar{v})} -1)}, \;\; \phi_2 = - \phi_3 =-\frac{ e^{ \bar{u}-\bar{v}}}{2(e^{ \bar{u}-\bar{v}} +1)}.
	\end{equation*}
	Hence,
	$$
	\ds \phi_1+\phi_2\frac{\sqrt {\bar G}}{\sqrt {\bar E}} = -\frac{ e^{ \bar{u}-\bar{v}}}{e^{ \bar{u}-\bar{v}} -1}; \quad 
	\ds \phi_3\frac{\sqrt {\bar E}}{\sqrt {\bar G}}+\phi_4 =\frac{ e^{ \bar{u}-\bar{v}}}{e^{ \bar{u}-\bar{v}} -1}, 
	$$
	and we can calculate the functions $\varphi(\bar u)$ and $\psi(\bar v)$ given by \eqref{eq:phi-psi}, which for this example  have the following form:
	\begin{equation*}
		\begin{array}{l}
			\vspace{2mm}
			\varphi(\bar u) = \ds{\sqrt{2} \,(e^{ \bar{u}_0-\bar{v}_0} -1) e^{c_1}};\\
			\vspace{2mm}
			\psi(\bar v) = \ds{\sqrt{2} \,(e^{ \bar{u}_0-\bar{v}_0} -1) e^{c_2}}.
		\end{array}
	\end{equation*}
	 For fixed $(\bar{u}_0, \bar{v}_0)$ we choose
	\begin{equation*}
		c_1=c_2 = -\ln(\sqrt{2} \,(e^{ \bar{u}_0-\bar{v}_0} -1))
	\end{equation*}
	and obtain $\varphi(\bar u)=\psi(\bar v)=1$, which means that the parameters $(\bar u, \bar v)$ are canonical.

\vskip 6mm 
\textbf{Acknowledgments:}
The first author is partially supported by the Bulgarian Ministry of Education and  Science, Scientific Programme ''Enhancing the Research Capacity in Mathematical Sciences (PIKOM)''  (contract DO1-67/05.05.2022) and by the Ministry of Education, Science and Technological Development of the Republic of Serbia (contract reg. no. 451-03-34/2026-03/200123).
The second author is partially supported by the National Science Fund, Ministry of Education and Science of Bulgaria, contract KP-06-N82/6.

\end{document}